\documentclass[12pt]{article}

\usepackage{amsfonts,amssymb,amsmath,amsthm,epsfig,euscript}

\usepackage{tikz}
\usetikzlibrary{matrix,trees,arrows}

\usetikzlibrary{positioning}
\usetikzlibrary{fit}
\usetikzlibrary{patterns}

\newtheorem{theorem}{Theorem}

\newtheorem{proposition}{Proposition}

\theoremstyle{definition}
{

}

\long\def\symbolfootnote[#1]#2{\begingroup
\def\thefootnote{\fnsymbol{footnote}}\footnote[#1]{#2}\endgroup}

\newcommand{\red}[1][\sigma]{\mathrm{red}(#1)}

\newcommand{\sg}{\sigma}

\newcommand{\mmp}{\mathrm{mmp}}

\def\A{\mathcal{A}}

\newcommand{\fig}[2]{\begin{figure}[ht]
\centerline{\scalebox{.66}{\epsfig{file=#1.eps}}}
\caption{#2}
\label{fig:#1}
\end{figure}}

\newcommand{\shadetheboxes}[1]{
	\foreach \x/\y in {#1}
      	\fill[pattern color = black!65, pattern=north east lines] (\x,\y) rectangle +(1,1);
	}
	
\newcommand{\drawthegrid}[1]{
	\draw (0.01,0.01) grid (#1+0.99,#1+0.99);
	}
	
\newcommand{\drawtheclpattern}[1]{
	\foreach \x/\y in {#1}
      	\filldraw (\x,\y) circle (6pt);
	}

\newcommand{\drawspecialbox}[1]{
	\foreach \x/\y/\z/\w/\A in {#1}
		{
       		\fill[color = white!100, opacity=1, rounded corners = 1.5pt] (\x+0.125,\y+0.125) rectangle (\z-0.125,\w-0.125);
       		\draw[color = black, rounded corners = 1.5pt] (\x+0.125,\y+0.125) rectangle (\z-0.125,\w-0.125);
       		\fill[black] (\x/2+\z/2,\y/2+\w/2) node {$\scriptstyle\A$};
       	}
    }

\newcommand{\mmpattern}[5]{									
  \raisebox{0.6ex}{
  \begin{tikzpicture}[scale=0.35, baseline=(current bounding box.center), #1]
  \useasboundingbox (0.0,-0.1) rectangle (#2+1.4,#2+1.1);
    
    \shadetheboxes{#4}
    
    \drawthegrid{#2}
    
    \drawspecialbox{#5}
    
    \drawtheclpattern{#3}

  \end{tikzpicture}}
}

\title{Quadrant marked mesh patterns in alternating permutations}

\author{
Sergey Kitaev \\
\small School of Computer and Information Sciences\\[-0.8ex]
\small University of Strathclyde\\[-0.8ex]
\small Glasgow G1 1XH, United Kingdom\\[-0.8ex]
\small \texttt{sergey.kitaev@cis.strath.ac.uk}
\and
Jeffrey Remmel \\
\small Department of Mathematics\\[-0.8ex]
\small University of California, San Diego\\[-0.8ex]
\small La Jolla, CA 92093-0112. USA\\[-0.8ex]
\small \texttt{jremmel@ucsd.edu}
}

\date{\small Submitted: Date 1;  Accepted: Date 2;
 Published: Date 3.\\
\small MR Subject Classifications: 05A15, 05E05}

\begin{document}
\maketitle

\begin{abstract}
\noindent \

This paper is continuation of the systematic study of distribution of 
quadrant marked mesh patterns initiated in \cite{kitrem}. 
We study quadrant marked mesh patterns on 
up-down and down-up permutations, also known as alternating and reverse alternating permutations, respectively.  In particular, we refine classic enumeration results of Andr\'{e} \cite{Andre1,Andre2} on alternating permutations by showing that the distribution of the quadrant marked mesh pattern of interest is given by $(\sec(xt))^{1/x}$ on up-down permutations of even length and by $\int_0^t (\sec(xz))^{1+\frac{1}{x}}dz$ on down-up permutations of odd length. \\

\end{abstract}

\section{Introduction}

The notion of mesh patterns was introduced by Br\"and\'en and Claesson \cite{BrCl} to provide explicit expansions for certain permutation statistics as, possibly infinite, linear combinations of (classical) permutation patterns (see \cite{kit} for a comprehensive introduction to the theory of permutation patterns).  This notion was further studied in \cite{AKV,HilJonSigVid,kitrem,kitremtie1,kitremtie2,Ulf}. 

 Let $\sigma = \sg_1 \ldots \sg_n$ be a permutation in 
the symmetric group $S_n$ 
written in one-line notation. Then we will consider the 
graph of $\sg$, $G(\sg)$, to be the set of points $(i,\sg_i)$ for 
$i =1, \ldots, n$.  For example, the graph of the permutation 
$\sg = 471569283$ is pictured in Figure~\ref{fig:basic}.  Then if we draw a coordinate system centered at a 
point $(i,\sg_i)$, we will be interested in  the points that 
lie in the four quadrants I, II, III, and IV of that 
coordinate system as pictured 
in Figure \ref{fig:basic}.  For any $a,b,c,d \in  
\mathbb{N}$ where $\mathbb{N} = \{0,1,2, \ldots \}$ is the set of 
natural numbers and any $\sg = \sg_1 \ldots \sg_n \in S_n$, 
we say that $\sg_i$ matches the 
quadrant marked mesh pattern $MMP(a,b,c,d)$ in $\sg$ if in $G(\sg)$  relative 
to the coordinate system which has the point $(i,\sg_i)$ as its  
origin,  there are 
$\geq a$ points in quadrant I, 
$\geq b$ points in quadrant II, $\geq c$ points in quadrant 
III, and $\geq d$ points in quadrant IV.  
For example, 
if $\sg = 471569283$, the point $\sg_4 =5$  matches 
the quadrant marked mesh pattern $MMP(2,1,2,1)$ since relative 
to the coordinate system with origin $(4,5)$,  
there are 3 points in $G(\sg)$ in quadrant I, 
1 point in $G(\sg)$ in quadrant II, 2 points in $G(\sg)$ in quadrant III, and 2 points in $G(\sg)$ in 
quadrant IV.  Note that if a coordinate 
in $MMP(a,b,c,d)$ is 0, then there is no condition imposed 
on the points in the corresponding quadrant. In addition, one can 
consider patterns  $MMP(a,b,c,d)$ where 
$a,b,c,d \in \mathbb{N} \cup \{\emptyset\}$. Here when 
one of the parameters $a$, $b$, $c$, or $d$ in  
$MMP(a,b,c,d)$ is the empty set, then for $\sg_i$ to match  
$MMP(a,b,c,d)$ in $\sg = \sg_1 \ldots \sg_n \in S_n$, 
it must be the case that there are no points in $G(\sg)$ relative 
to coordinate system with origin $(i,\sg_i)$ in the corresponding 
quadrant. For example, if $\sg = 471569283$, the point 
$\sg_3 =1$ matches 
the marked mesh pattern $MMP(4,2,\emptyset,\emptyset)$ since relative 
to the coordinate system with origin $(3,1)$, 
there are 6 points in $G(\sg)$ in quadrant I, 
2 points in $G(\sg)$ in quadrant II, no  points in $G(\sg)$ in quadrant III, and no  points in $G(\sg)$ in quadrant IV.  We let 
$\mmp^{(a,b,c,d)}(\sg)$ denote the number of $i$ such that 
$\sg_i$ matches the marked mesh pattern $MMP(a,b,c,d)$ in $\sg$.

\fig{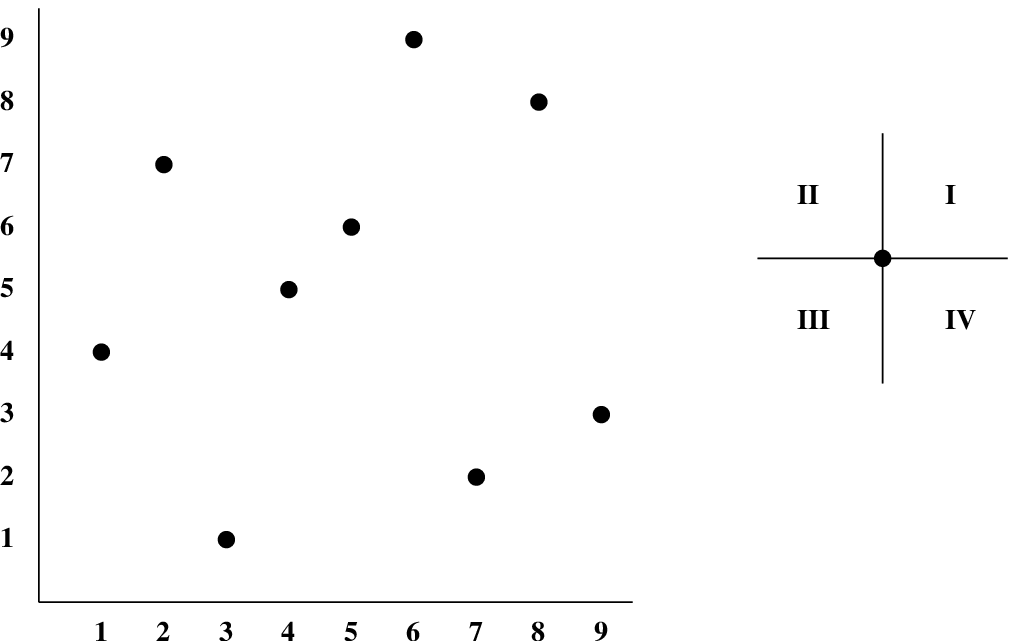}{The graph of $\sg = 471569283$.}

Note how the (two-dimensional) notation of \'Ulfarsson \cite{Ulf} for marked mesh patterns corresponds to our (one-line) notation for quadrant marked mesh patterns. For example,

\[
MMP(0,0,k,0)=\mmpattern{scale=2.3}{1}{1/1}{}{0/0/1/1/k}\hspace{-0.25cm},\  MMP(k,0,0,0)=\mmpattern{scale=2.3}{1}{1/1}{}{1/1/2/2/k}\hspace{-0.25cm},
\]

\[
MMP(0,a,b,c)=\mmpattern{scale=2.3}{1}{1/1}{}{0/1/1/2/a} \hspace{-2.07cm} \mmpattern{scale=2.3}{1}{1/1}{}{0/0/1/1/b} \hspace{-2.07cm} \mmpattern{scale=2.3}{1}{1/1}{}{1/0/2/1/c} \ \mbox{ and }\ \ \ MMP(0,0,\emptyset,k)=\mmpattern{scale=2.3}{1}{1/1}{0/0}{1/0/2/1/k}\hspace{-0.25cm}.
\]

Kitaev and Remmel \cite{kitrem} studied the distribution of 
quadrant marked mesh patterns in the symmetric group $S_n$ 
and Kitaev, Remmel, and 
Tiefenbruck \cite{kitremtie1,kitremtie2} 
studied the distribution of quadrant marked mesh patterns 
in $132$-avoiding permutations in $S_n$.
The main goal of this paper is to study the distribution 
of the statistics $\mmp^{(1,0,0,0)}$, $\mmp^{(0,1,0,0)}$, $\mmp^{(0,0,1,0)}$, 
and $\mmp^{(0,0,0,1)}$ in the set of {\em up-down} and {\em down-up permutations}. 
We say 
that $\sg= \sg_1 \ldots \sg_n \in S_n$ is an up-down permutation if 
it is of the form
$$\sg_1 < \sg_2 > \sg_3< \sg_4 > \sg_5 < \cdots,$$ 
and $\sg$ is a down-up permutation if it is of the form
$$\sg_1 > \sg_2 < \sg_3 > \sg_4 < \sg_5 > \cdots. $$ 
Let  $UD_n$ denote the set of all up-down permutations in $S_n$ and 
$DU_n$ denote the set of all down-up permutations in $S_n$. 
Given a permutation $\sg = \sg_1 \ldots \sg_n \in S_n$, 
we define the reverse of $\sg$, $\sg^r$, to be 
$\sg_n \sg_{n-1} \ldots \sg_1$ and the complement of 
$\sg$, $\sg^c$, to be $(n+1-\sg_1)(n+1-\sg_2) \ldots (n+1 -\sg_n)$. 
For $n \geq 1$, we let 
\begin{eqnarray*}
A_{2n}(x) &=& \sum_{\sg \in UD_{2n}}x^{\mmp^{(1,0,0,0)}(\sg)}, \ \ \ \ \ \ \ 
B_{2n-1}(x) = \sum_{\sg \in UD_{2n-1}} x^{\mmp^{(1,0,0,0)}(\sg)}, \\
C_{2n}(x) &=& \sum_{\sg \in DU_{2n}}x^{\mmp^{(1,0,0,0)}(\sg)}, \ \mbox{and} \ 
D_{2n-1}(x) = \sum_{\sg \in DU_{2n-1}}x^{\mmp^{(1,0,0,0)}(\sg)}. 
\end{eqnarray*}
We then have the following simple proposition. 
\begin{proposition}\label{prop1} For all $n \geq 1$, 

\begin{itemize}
\item[\rm{(1)}] $\displaystyle  A_{2n}(x) = 
\sum_{\sg \in DU_{2n}} x^{\mmp^{(0,1,0,0)}(\sg)}  
= \sum_{\sg \in DU_{2n}}x^{\mmp^{(0,0,0,1)}(\sg)} 
= \sum_{\sg \in UD_{2n}}x^{\mmp^{(0,0,1,0)}(\sg)}$, 
\item[\rm{(2)}] $\displaystyle 
C_{2n}(x) = \sum_{\sg \in UD_{2n}}x^{\mmp^{(0,1,0,0)}(\sg)} 
= \sum_{\sg \in UD_{2n}}x^{\mmp^{(0,0,0,1)}(\sg)} = 
\sum_{\sg \in DU_{2n}}x^{\mmp^{(0,0,1,0)}(\sg)}$, 
\item[\rm{(3)}] $\displaystyle
B_{2n-1}(x) = \sum_{\sg \in UD_{2n-1}}x^{\mmp^{(0,1,0,0)}(\sg)} 
=  \sum_{\sg \in DU_{2n-1}}x^{\mmp^{(0,0,0,1)} (\sg)}
= \sum_{\sg \in DU_{2n-1}}x^{\mmp^{(0,0,1,0)}(\sg)}$, and 
\item[\rm{(4)}] $\displaystyle D_{2n-1}(x) = \sum_{\sg \in DU_{2n-1}}x^{\mmp^{(0,1,0,0)}(\sg)} 
= \sum_{\sg \in UD_{2n-1}}x^{\mmp^{(0,0,0,1)}(\sg)} 
= \sum_{\sg \in UD_{2n-1}}x^{\mmp^{(0,0,1,0)}(\sg)}$. 
\end{itemize}
\end{proposition}
\begin{proof}
It is easy to see that for any 
$\sg \in S_n$, $$\mmp^{(1,0,0,0)}(\sg)=\mmp^{(0,1,0,0)}(\sg^r)= 
\mmp^{(0,0,0,1)}(\sg^c)= \mmp^{(0,0,1,0)}((\sg^r)^c).$$ 
Then part 1 easily follows since   
$$ \sg \in UD_{2n} \iff \sg^r \in DU_{2n} \iff \sg^c \in DU_{2n} \iff 
(\sg^r)^c \in UD_{2n}.$$ 

Parts 2, 3, and 4 are proved in a similar manner. 
\end{proof}

It follows from Propostion \ref{prop1} that the study of the distribution 
of the statistics $\mmp^{(1,0,0,0)}$, $\mmp^{(0,1,0,0)}$, $\mmp^{(0,0,1,0)}$, 
and $\mmp^{(0,0,0,1)}$ in the set of up-down and down-up permutations 
can be reduced to the study of the following generating functions: 
\begin{eqnarray*}
A(t,x) &=& 1+ \sum_{n\geq 1} A_{2n}(x) \frac{t^{2n}}{(2n)!}, \\
B(t,x) &=& \sum_{n\geq 1} B_{2n-1}(x)\frac{t^{2n-1}}{(2n-1)!}, \\
C(t,x) &=& 1+ \sum_{n\geq 1} C_{2n}(x) \frac{t^{2n}}{(2n)!}, \ \mbox{and} \\
D(t,x) &=& \sum_{n\geq 1} D_{2n-1}(x) \frac{t^{2n-1}}{(2n-1)!}.
\end{eqnarray*} 

In the case when $x=1$, these generating functions are well known. 
That is, the operation of complementation shows that 
 $A_{2n}(1) = C_{2n}(1)$ and $B_{2n-1}(1) = D_{2n-1}(1)$ for 
all $n \geq 1$ and Andr\'{e} \cite{Andre1,Andre2}
proved that
\begin{equation*}
\sum_{n\geq 0} A_{2n}(1) \frac{t^{2n}}{(2n)!} = \sec(t)
\end{equation*}
and 
\begin{equation*}
\sum_{n\geq 1} B_{2n-1}(1) \frac{t^{2n-1}}{(2n-1)!} = \tan(t).
\end{equation*}
Thus, the number of up-down permutations is given by  the following exponential generating function 
\begin{equation}\label{cool}
\sec(t)+\tan(t)=\tan\left(\frac{t}{2}+\frac{\pi}{4}\right).
\end{equation}

We shall prove the following theorem.

\begin{theorem}\label{thm:main} We have
\begin{eqnarray*}
A(t,x) &=& (\sec(xt))^{1/x}, \\
B(t,x) &=& (\sec(xt))^{1/x} \int_0^t (\sec(xz))^{-1/x} dz, \\
C(t,x)&=& 1 + \int_0^t (\sec(xy))^{1+\frac{1}{x}}\int_0^y (\sec(xz))^{1/x} dz\ dy, \ \mbox{and}\\
D(t,x) &=& \int_0^t (\sec(xz))^{1+\frac{1}{x}}dz.
\end{eqnarray*}
\end{theorem}

As an immediate corollary to Theorem \ref{thm:main} we get, for example, that the number of up-down permutations by occurrences of $MMP(1,0,0,0)$ is given by
$$A(t,x)+B(t,x)=(\sec(xt))^{1/x} \left(1+\int_0^t (\sec(xz))^{-1/x} dz\right)$$
which refines  (\ref{cool}).

One can use these generating functions to find some initial values 
of the polynomials $A_{2n}(x)$, $B_{2n-1}(x)$, $C_{2n}(x)$,  and 
$D_{2n-1}(x)$. For example, we have used Mathematica to compute 
 the following tables. \\
\ \\
\begin{tabular}{|l|l|}
\hline
$n$ &  $A_{2n}(x)$ \\
\hline
0 & 1 \\
\hline
1 & x\\
\hline
2 & $x^2 (3+2 x)$ \\
\hline
3 & $x^3 \left(15+30 x+16 x^2\right)$\\
\hline
4 & $x^4 \left(105+420 x+588 x^2+272 x^3\right)$\\
\hline
5 & $x^5 \left(945+6300 x+16380 x^2+18960
x^3+7936 x^4\right)$\\
\hline
6 & $x^6 \left(10395+103950 x+429660 x^2+893640 x^3+911328 x^4+
353792 x^5\right)$\\
\hline
\end{tabular} \\
\ \\
\ \\
\begin{tabular}{|l|l|}
\hline
$n$ &  $B_{2n-1}(x)$ \\
\hline 
1& 1 \\
\hline
2 & $2 x$\\
\hline
3 & $8 x^2 (1+x)$ \\
\hline 
4 & $16 x^3 \left(3+8 x+6 x^2\right)$ \\
\hline
5& $128 x^4 \left(3+15 x+27 x^2+17 x^3\right)$ \\
\hline 
6 & $256 x^5 \left(15+120 x+381
x^2+556 x^3+310 x^4\right)$\\
\hline
7 & $1024 x^6 \left(45+525 x+2562 x^2+6420 x^3+8146 x^4+4146 x^5\right)$\\
\hline
\end{tabular}\\
\ \\
\ \\
\begin{tabular}{|l|l|}
\hline
$n$ &  $C_{2n}(x)$ \\
\hline
0 & 1 \\
\hline
1 & 1 \\
\hline 
2 & $x (2+3 x)$\\
\hline 
3 & $x^2 \left(8+28 x+25 x^2\right)$ \\
\hline 
4 & $x^3 \left(48+296 x+614 x^2+427 x^3\right)$ \\
\hline 
5 & $x^4 \left(384+3648 x+13104 x^2+20920
x^3+12465 x^4\right)$\\
\hline 
6 & $x^5 \left(3840+51840 x+282336 x^2+769072 x^3+
1039946 x^4+555731 x^5\right)$
\\
\hline 
\end{tabular}\\
\ \\
\ \\
\begin{tabular}{|l|l|}
\hline
$n$ &  $D_{2n-1}(x)$ \\
\hline
1 & 1 \\
\hline 
2 & $x (1+x)$ \\
\hline 
3 & $x^2 \left(3+8 x+5 x^2\right)$ \\
\hline 
4 & $x^3 \left(15+75 x+121 x^2+61 x^3\right)$ \\
\hline 
5 & $x^4 \left(105+840 x+2478 x^2+3128 x^3+1385
x^4\right)$ \\
\hline 
6 & $x^5 \left(945+11025 x+51030 x^2+115350 x^3+124921 x^4+50521 x^5\right)$ \\
\hline 
7 & $x^6 \left(10395+166320 x+1105335 x^2+3859680 x^3+7365633 x^4+7158128
x^5+2702765 x^6\right)$\\
\hline
\end{tabular}
\ \\

The outline of this paper is as follows.  In Section 2, we shall 
prove Theorem \ref{thm:main}. Then in Section 3, we shall study the entries of the tables above explaining them either explicitly or through recursions. 

\section{Proof of Theorem \ref{thm:main}}

The proof of all parts of Theorem \ref{thm:main} proceed 
in the same manner. That is, there are simple  
recursions satisfied by the polynomials 
$A_{2n}(x)$, $B_{2n+1}(x)$, $C_{2n}(x)$, and $D_{2n+1}(x)$ based on the 
position of the largest value in the permutation. 

\subsection{The generating function $A(t,x)$}

If $\sg = \sg_1 \ldots \sg_{2n} \in UD_{2n}$, then  
$2n$ must occur in one of the positions $2,4, \ldots , 2n$.  Let 
$UD_{2n}^{(2k)}$ denote the set of permutations $\sg \in UD_{2n}$ 
such that $\sg_{2k} = 2n$.  A schematic diagram of an element
in $UD_{2n}^{(2k)}$ is pictured in Figure \ref{fig:ud2k2n}.

\fig{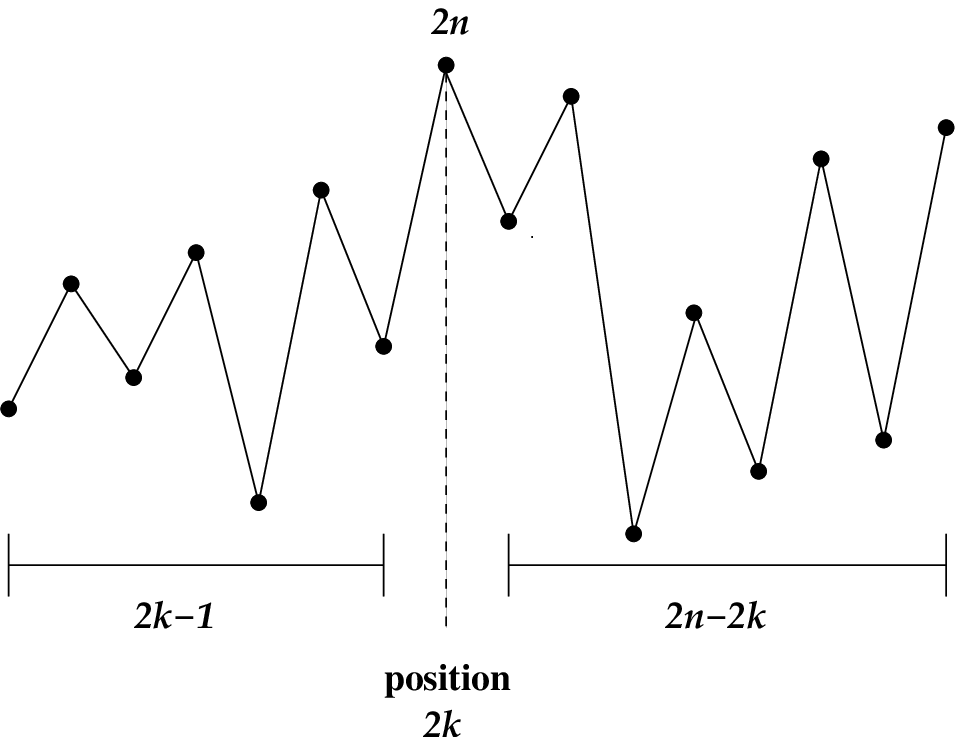}{The graph of a $\sg \in UD_{2n}^{(2k)}$.}

Note that there are $\binom{2n-1}{2k-1}$ ways 
to pick the elements which occur to the left of position $2k$ in such 
$\sg$ and there are $B_{2k-1}(1)$ ways to order them since 
the elements to the left of position $2k$ form an up-down permutation 
of length $2k-1$. Each of the elements to the left of position 
$2k$ contributes to $\mmp^{(1,0,0,0)}(\sg)$.  Thus the contribution 
of the elements to the left of position $2k$ in 
$\sum_{\sg \in UD_{2n}^{(2k)}} x^{\mmp^{(1,0,0,0)}(\sg)}$ is 
$B_{2k-1}(1)x^{2k-1}$.  There are 
$A_{2n-2k}(1)$ ways to order the elements to 
the right of position $2k$ since they must form an 
up-down permutation  of length $2n-2k$. Since the elements 
to the left of position $2k$ have no effect on whether an element 
to the right of position $2k$ contributes to $\mmp^{(1,0,0,0)}(\sg)$, it 
follows that the contribution 
of the elements to the right of position $2k$ in 
$\sum_{\sg \in UD_{2n}^{(2k)}} x^{\mmp^{(1,0,0,0)}(\sg)}$ is 
$A_{2n-2k}(x)$. It thus follows that 
\begin{equation*}
A_{2n}(x) = \sum_{k=1}^n \binom{2n-1}{2k-1} B_{2k-1}(1)x^{2k-1} A_{2n-2k}(x)
\end{equation*}
or, equivalently,  
\begin{equation}\label{Arec2}
\frac{A_{2n}(x)}{(2n-1)!} = \sum_{k=1}^n \frac{B_{2k-1}(1)x^{2k-1}}{(2k-1)!} 
\frac{A_{2n-2k}(x)}{(2n-2k)!}.
\end{equation}
Multiplying both sides of (\ref{Arec2}) by $t^{2n-1}$ and summing 
for $n \geq 1$, we see that 
$$\sum_{n \geq 1} \frac{A_{2n}(x)t^{2n-1}}{(2n-1)!} = 
\left(\sum_{n \geq 1}  \frac{B_{2n-1}(1)x^{2n-1}t^{2n-1}}{(2n-1)!}\right) 
\left(\sum_{n \geq 0} \frac{A_{2n}(x)t^{2n}}{(2n)!}\right).$$ 
By Andr\'e's result, 
$$\sum_{n \geq 1}  \frac{B_{2n-1}(1)x^{2n-1}t^{2n-1}}{(2n-1)!} = 
\tan(xt)$$ 
so that 
\begin{equation*}\label{Adiff}
\frac{\partial}{\partial t} A(t,x) = \tan(xt) A(t,x).
\end{equation*}
Our initial condition is that $A(0,x)=1$.  It is easy to check 
that the solution to this differential equation is 
\begin{equation*}\label{Afin}
A(t,x) = (\sec(xt))^{1/x}.
\end{equation*}

\subsection{The generating function $B(t,x)$}

If $\sg = \sg_1 \ldots \sg_{2n+1} \in UD_{2n+1}$, then  
$2n+1$ must occur in one of the positions $2,4, \ldots , 2n$.  Let 
$UD_{2n+1}^{(2k)}$ denote the set of permutations $\sg \in UD_{2n+1}$ 
such that $\sg_{2k} = 2n+1$.  A schematic diagram of an element
in $UD_{2n}^{(2k)}$ is pictured in Figure \ref{fig:ud2k2n+1}. 

\fig{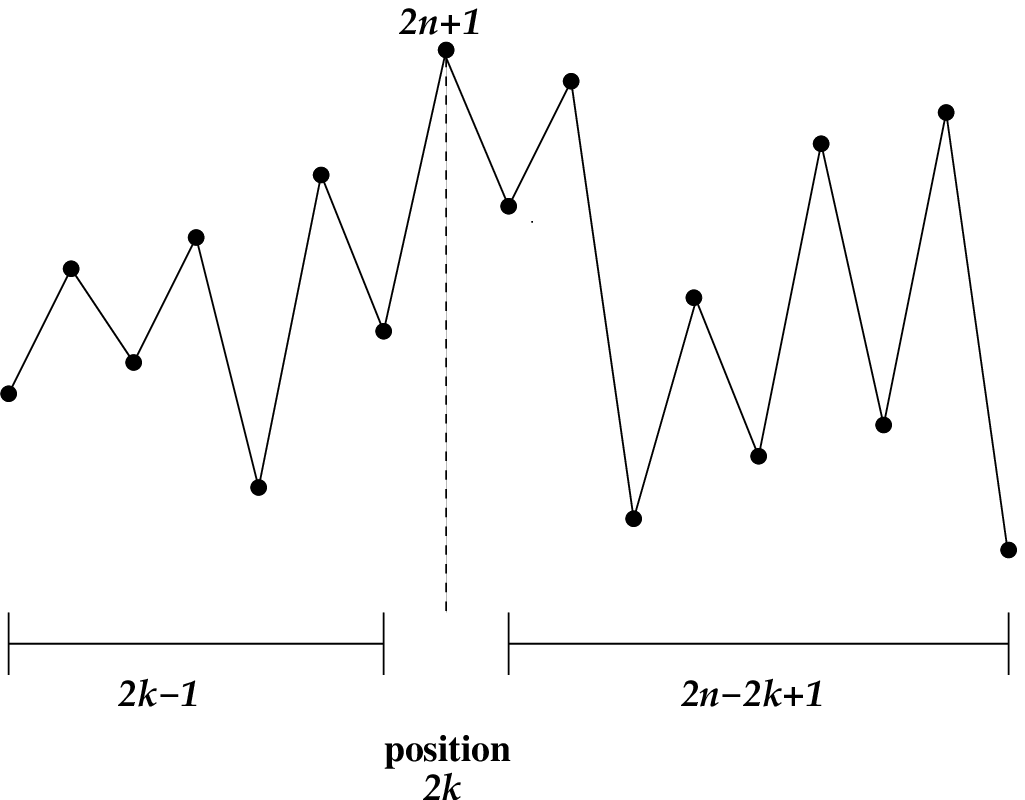}{The graph of a $\sg \in UD_{2n+1}^{(2k)}$.}

Again there are $\binom{2n}{2k-1}$ ways 
to pick the elements which occur to the left of position $2k$ in such 
$\sg$ and  the contribution 
of the elements to the left of position $2k$ in 
$\sum_{\sg \in UD_{2n+1}^{(2k)}} x^{\mmp^{(1,0,0,0)}(\sg)}$ is 
$B_{2k-1}(1)x^{2k-1}$.  There are 
$B_{2n-2k+1}(1)$ ways to order the elements to 
the right of position $2k$ since they must form an 
up-down permutation  of length $2n-2k+1$. Since the elements 
to the left of position $2k$ have no effect on whether an element 
to the right of position $2k$ contributes to $\mmp^{(1,0,0,0)}(\sg)$, it 
follows that the contribution 
of the elements to the right of position $2k$ in 
$\sum_{\sg \in UD_{2n+1}^{(2k)}} x^{\mmp^{(1,0,0,0)}(\sg)}$ is 
$B_{2n-2k+1}(x)$. It thus follows that if $n \geq 1$, then 
\begin{equation*}\label{Brec1}
B_{2n+1}(x) = \sum_{k=1}^n \binom{2n}{2k-1} B_{2k-1}(1)x^{2k-1} B_{2n-2k+1}(x).
\end{equation*}
Hence for $n\geq 1$, 
\begin{equation}\label{Brec2}
\frac{B_{2n+1}(x)}{(2n)!} = \sum_{k=1}^n \frac{B_{2k-1}(1)x^{2k-1}}{(2k-1)!} 
\frac{B_{2n-2k+1}(x)}{(2n-2k+1)!}.
\end{equation}
Multiplying both sides of (\ref{Brec2}) by $t^{2n}$, summing 
for $n \geq 1$, and taking into account that $B_1(x)=1$, we see that 
$$\sum_{n \geq 0} \frac{B_{2n+1}(x)t^{2n}}{(2n)!} = 1+ 
\left(\sum_{n \geq 0}  \frac{B_{2n+1}(1)x^{2n+1}t^{2n+1}}{(2n+1)!}\right) 
\left(\sum_{n \geq 0} \frac{B_{2n+1}(x)t^{2n+1}}{(2n+1)!}\right).$$ 
Since  
$$\sum_{n \geq 1}  \frac{B_{2n-1}(1)x^{2n-1}t^{2n-1}}{(2n-1)!} = 
\tan(xt),$$ 
we see that  
\begin{equation*}\label{Bdiff}
\frac{\partial}{\partial t} B(t,x) = 1+\tan(xt) B(t,x).
\end{equation*}
Our initial condition is that $B(0,x)=0$.  It is easy to check 
that the solution to this differential equation is 
\begin{equation*}\label{Bfin}
B(t,x) = (\sec(xt))^{1/x}\int_0^t (\sec(xz))^{-1/x}dz.
\end{equation*}

\subsection{The generating function $C(t,x)$}

If $\sg = \sg_1 \ldots \sg_{2n} \in DU_{2n}$, then  
$2n$ must occur in one of the positions $1,3, \ldots , 2n-1$.  Let 
$DU_{2n}^{(2k+1)}$ denote the set of permutations $\sg \in DU_{2n}$ 
such that $\sg_{2k+1} = 2n$.  A schematic diagram of an element
in $DU_{2n}^{(2k+1)}$ is pictured in Figure \ref{fig:du2k2n}.

\fig{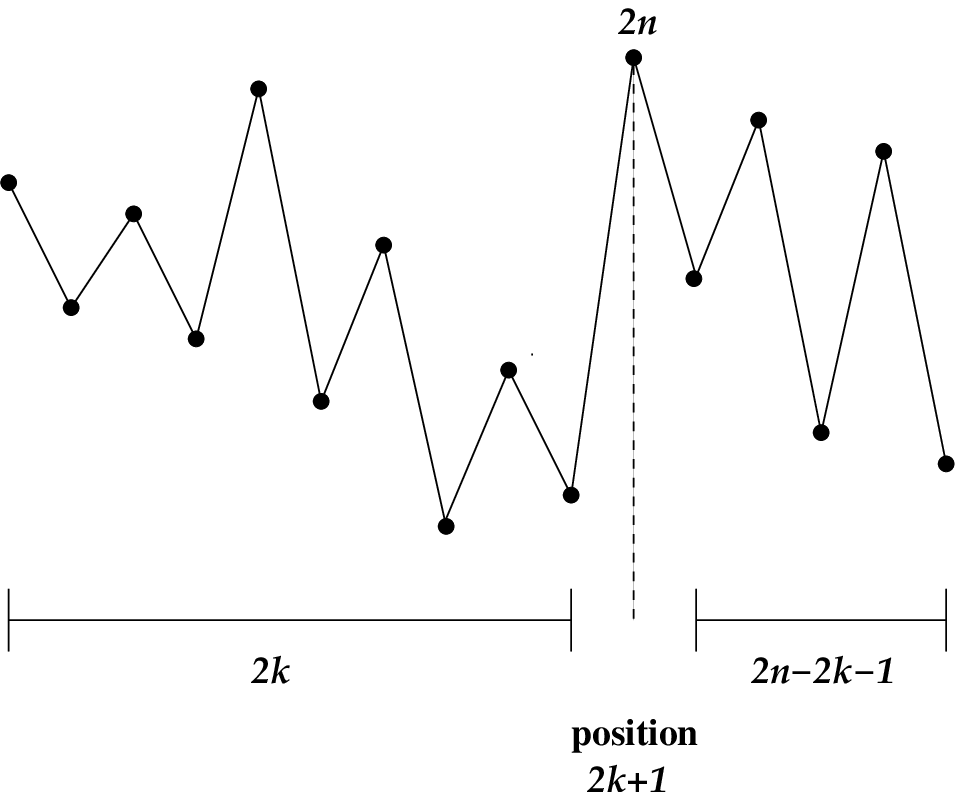}{The graph of a $\sg \in DU_{2n}^{(2k+1)}$.}

Note that there are $\binom{2n-1}{2k}$ ways 
to pick the elements which occur to the left of position $2k+1$ in such 
$\sg$ and there are $C_{2k}(1)=A_{2k}(1)$ ways to order them since 
the elements to the left of position $2k+1$ form a down-up permutation 
of length $2k$.  Each of the elements to the left of position 
$2k+1$ contributes to $\mmp^{(1,0,0,0)}(\sg)$.  Thus the contribution 
of the elements to the left of position $2k+1$ in 
$\sum{\sg \in UD_{2n}^{(2k)}} x^{\mmp^{(1,0,0,0)}(\sg)}$ is 
$A_{2k}(1)x^{2k}$.  There are 
$B_{2n-2k-1}(1)$ ways to order the elements to 
the right of position $2k+1$ since they must form an 
up-down permutation  of length $2n-2k+1$. Since the elements 
to the left of position $2k+1$ have no effect on whether an element 
to the right of position $2k+1$ contributes to $\mmp^{(1,0,0,0)}(\sg)$, it 
follows that the contribution 
of the elements to the right of position $2k$ in 
$\sum_{\sg \in UD_{2n}^{(2k)}} x^{\mmp^{(1,0,0,0)}(\sg)}$ is 
$B_{2n-2k-1}(x)$. It thus follows that 
\begin{equation*}\label{Crec1}
C_{2n}(x) = \sum_{k=0}^{n-1} \binom{2n-1}{2k} A_{2k}(1)x^{2k} B_{2n-2k-1}(x), 
\end{equation*} 
or, equivalently, 
\begin{equation}\label{Crec2}
\frac{C_{2n}(x)}{(2n-1)!} = \sum_{k=0}^{n-1} \frac{A_{2k}(1)x^{2k}}{(2k)!} 
\frac{B_{2n-2k-1}(x)}{(2n-2k-1)!}.
\end{equation}
Multiplying both sides of (\ref{Crec2}) by $t^{2n-1}$ and summing 
for $n \geq 1$, we see that 
$$\sum_{n \geq 1} \frac{C_{2n}(x)t^{2n-1}}{(2n-1)!} = 
\left(\sum_{n \geq 1}  \frac{B_{2n-1}(1)x^{2n-1}t^{2n-1}}{(2n-1)!}\right) 
\left(\sum_{n \geq 0} \frac{A_{2n}(x)t^{2n}}{(2n)!}\right).$$ 
By Andr\'e's result, 
$$\sum_{n \geq 0}  \frac{A_{2n}(1)x^{2n}t^{2n}}{(2n)!} = 
\sec(xt)$$ 
so that 
\begin{equation}\label{Cdiff}
\frac{\partial}{\partial t} C(t,x) = \sec(xt) B(t,x) = 
(\sec(xt))^{1+\frac{1}{x}} \int_0^t (\sec(xz))^{\frac{-1}{x}} dz.
\end{equation}
Our initial condition is that $C(0,x)=1$.  Both Maple and Mathematica 
will solve this differential equation but the final expressions  
are complicated and not particularly useful for enumeration 
purposes. Thus we actually used the RHS of  (\ref{Cdiff}) to find 
the entries of the table for the initial values of 
$C_{2n}(x)$ given in the introduction. Nevertheless, we can 
record the solution of (\ref{Cdiff}) as  
\begin{equation*}\label{Cfin}
C(t,x) = 1+ \int_0^t (\sec(xy))^{1+\frac{1}{x}} 
\int_0^y (\sec(xz))^{\frac{-1}{x}} dz\ dy.
\end{equation*}

\subsection{The generating function $D(t,x)$}

If $\sg = \sg_1 \ldots \sg_{2n+1} \in DU_{2n+1}$, then  
$2n+1$ must occur in one of the positions $1,3, \ldots , 2n+1$.  Let 
$DU_{2n+1}^{(2k+1)}$ denote the set of permutations $\sg \in DU_{2n+1}$ 
such that $\sg_{2k+1} = 2n+1$.  A schematic diagram of an element
in $DU_{2n+1}^{(2k+1)}$ is pictured in Figure \ref{fig:du2k2n+1}.

\fig{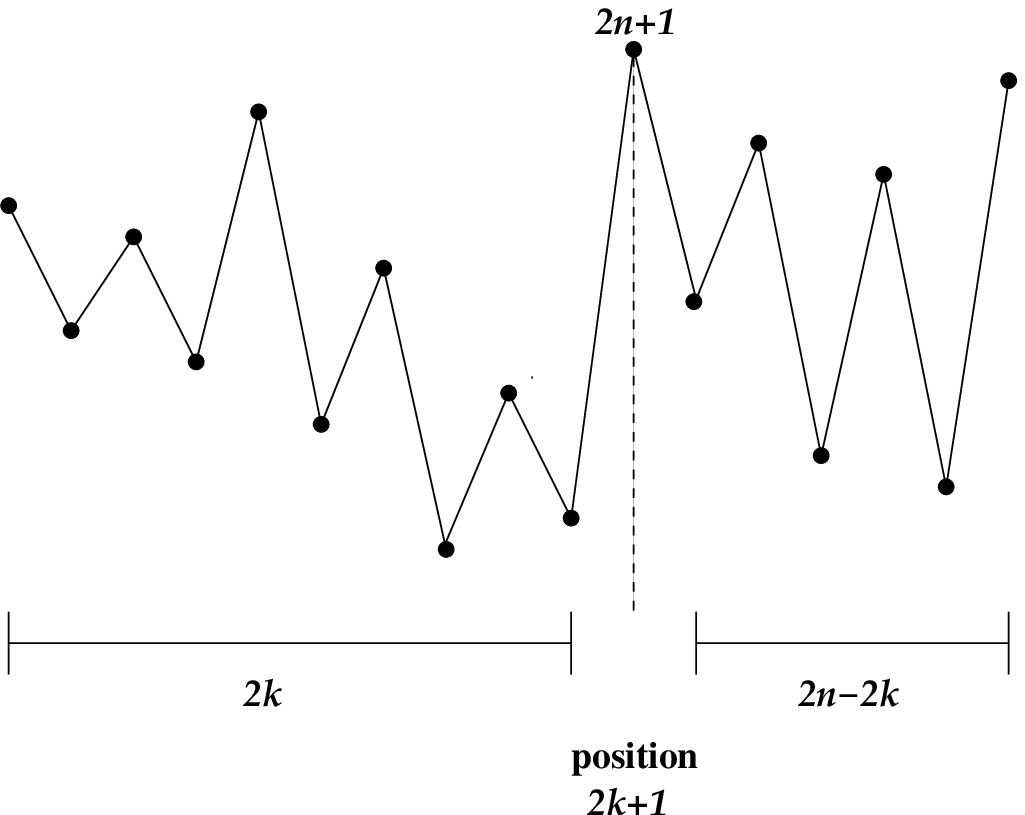}{The graph of a $\sg \in DU_{2n+1}^{(2k+1)}$.}

Note that there are $\binom{2n}{2k}$ ways 
to pick the elements which occur to the left of position $2k+1$ in such 
$\sg$ and there are $C_{2k}(1) = A_{2k}(1)$ ways to order them since 
the elements to the right of position $2k+1$ form a down-up permutation 
of length $2k$. Each of the elements to the left of position 
$2k+1$ contributes to $\mmp^{(1,0,0,0)}(\sg)$.  Thus the contribution 
of the elements to the left of position $2k+1$ in 
$\sum_{\sg \in DU_{2n+1}^{(2k+1)}} x^{\mmp^{(1,0,0,0)}(\sg)}$ is 
$A_{2k}(1)x^{2k}$.  There are 
$A_{2n-2k}(1)$ ways to order the elements to 
the right of position $2k+1$ since they must form an 
up-down permutation  of length $2n-2k$. Since the elements 
to the left of position $2k+1$ have no effect on whether an element 
to the right of position $2k+1$ contributes to $\mmp^{(1,0,0,0)}(\sg)$, it 
follows that the contribution 
of the elements to the right of position $2k+1$ in 
$\sum_{\sg \in DU_{2n+1}^{(2k+1)}} x^{\mmp^{(1,0,0,0)}(\sg)}$ is 
$A_{2n-2k}(x)$. It thus follows that if $n \geq 1$, then 
\begin{equation*}\label{Drec1}
D_{2n+1}(x) = \sum_{k=0}^n \binom{2n}{2k} A_{2k}(1)x^{2k} A_{2n-2k}(x).
\end{equation*}
Hence for $n\geq 1$, 
\begin{equation}\label{Drec2}
\frac{D_{2n+1}(x)}{(2n)!} = \sum_{k=0}^n \frac{A_{2k}(1)x^{2k}}{(2k)!} 
\frac{A_{2n-2k}(x)}{(2n-2k)!}.
\end{equation}
Multiplying both sides of (\ref{Drec2}) by $t^{2n}$ and summing 
for $n \geq 0$, we see that 
$$\sum_{n \geq 0} \frac{D_{2n+1}(x)t^{2n}}{(2n)!} =  
\left(\sum_{n \geq 0}  \frac{A_{2n}(1)x^{2n}t^{2n}}{(2n)!}\right) 
\left(\sum_{n \geq 0} \frac{A_{2n}(x)t^{2n}}{(2n)!}\right)$$ 
so that 
\begin{equation*}\label{Ddiff}
\frac{\partial}{\partial t} D(t,x) = \sec(x,t) A(t,x) = 
(\sec(xt))^{1+ \frac{1}{x}}.
\end{equation*}
Our initial condition is that $D(0,x)=0$ so that the solution to this differential equation is 
\begin{equation*}\label{Dfin}
D(t,x) = \int_0^t (\sec(xz))^{1+\frac{1}{x}}dz.
\end{equation*}

\subsection{A remark on $MMP(k,0,0,0)$ for $k \geq 2$}\label{sub2-5}

We note that we cannot apply the same techniques to 
find the distribution of marked mesh patterns $MMP(k,0,0,0)$ 
in up-down and down-up permutations when $k \geq 2$. 
For example, suppose that we try to develop a recursion 
for $A^{(2,0,0,0)}_{2n}(x) = \sum_{\sg \in UD_{2n}} 
x^{\mmp^{(2,0,0,0)}(\sg)}$. Then if we consider the permutations 
$\sg = \sg_1 \ldots \sg_{2n} \in UD_{2n}$ such that 
$\sg_{2k} =2n$, we still have $\binom{2n-1}{2k-1}$ ways to 
pick the elements for $\sg_1 \ldots \sg_{2k-1}$.  However, 
in this case the question of whether some $\sg_i$ with 
$i < 2k$ matches the marked mesh pattern $MMP(2,0,0,0)$ in $\sg$ 
is dependent on what values occur in $\sg_{2k+1} \ldots 
\sg_{2n}$. For example, if $2n-1 \in \{\sg_{2k+1}, \ldots , \sg_{2n}\}$, 
then every $\sg_i$ with $i \leq k$ will match the marked mesh 
pattern $MMP(2,0,0,0)$ in $\sg$. However, if $2n-1 \in 
\{\sg_1, \ldots , \sg_{2k-1}\}$, this will not be the case. Thus 
we cannot develop a simple recursion for 
$A^{(2,0,0,0)}_{2n}(x)$. 

\section{The coefficients of the polynomials $A_{2n}(x)$, $B_{2n+1}(x)$, $C_{2n}(x)$, and $D_{2n+1}(x)$.}

The main goal of this section is to explain several of the coefficients of the polynomials $A_{2n}(x)$, $B_{2n+1}(x)$, $C_{2n}(x)$, and $D_{2n+1}(x)$. 
First it is easy to understand the coefficients of the 
lowest power of $x$ in each of these polynomials. That is, 
we have the following theorem, where $0!!=1$ and, for $n \geq 1$, $(2n)!! = \prod_{i=1}^n (2i)$ and 
 $(2n-1)!! = \prod_{i=1}^n (2i-1)$. 

\begin{theorem}\label{thm:base} \ 

\begin{itemize} 
\item[\rm{(1)}] For all $n \geq 1$, 
\begin{equation*} 
A_{2n}(x)|_{x^k} = \begin{cases} 0 & \ \mbox{if $0 \leq k <n$} \\
(2n-1)!! & \ \mbox{if $k=n$}.
\end{cases}
\end{equation*}
\item[\rm{(2)}] For all $n \geq 1$, 
\begin{equation*} 
B_{2n+1}(x)|_{x^k} = \begin{cases} 0 & \ \mbox{if $0 \leq k < n$} \\
(2n)!! & \ \mbox{if $k=n$}.
\end{cases}
\end{equation*}
\item[\rm{(3)}] For all $n \geq 1$, 
\begin{equation*} 
C_{2n}(x)|_{x^k} = \begin{cases} 0 & \ \mbox{if $0 \leq k <n-1$} \\
(2(n-1))!! & \ \mbox{if $k=n-1$}.
\end{cases}
\end{equation*}
\item[\rm{(4)}] For all $n \geq 1$, 
\begin{equation*} 
D_{2n+1}(x)|_{x^k} = \begin{cases} 0 & \ \mbox{if $0 \leq k < n$} \\
(2n-1)!! & \ \mbox{if $k=n$}.
\end{cases}
\end{equation*}
\end{itemize}
\end{theorem}
\begin{proof}

For (1), note that if $\sg = \sg_1 \ldots \sg_{2n} \in UD_{2n}$, then 
$\sg_{2i+1}$ matches the pattern $MMP(1,0,0,0)$ for $i =0, \ldots, n-1$. 
Thus $mpp^{(1,0,0,0)}(\sg) \geq n$. We now proceed by induction 
to prove that $A_{2n}(x)|_{x^n} = (2n-1)!!$ for all $n \geq 1$. This 
is clear for $n=1$ since $A_2(x) = x$.
Now suppose that $\sg = \sg_1 \ldots \sg_{2n} \in UD_{2n}$ and  
$mpp^{(1,0,0,0)}(\sg)=n$. It is then easy to see that it 
must be the case that $\sg_2 = 2n$ (otherwise $\sg_2$ is an unwanted occurrence of the pattern $MMP(1,0,0,0)$). Moreover, if 
$\tau = \red[\sg_3 \ldots \sg_{2n}]$, then $\tau \in UD_{2n-2}$ and 
$mmp^{(1,0,0,0)}(\tau) =n-1$.  Thus since we are assuming by 
induction that $A_{2n-2}(x)|_{x^{n-1}} = (2n-3)!!$, we have $2n-1$ choices of $\sg_1$ and $(2n-3)!!$ choices for  
$\tau$.  Hence $A_{2n}(x)|_{x^{n}} = (2n-1)!!$.

For (2), note that if $\sg = \sg_1 \ldots \sg_{2n+1} \in UD_{2n+1}$, then 
$\sg_{2i+1}$ matches the pattern $MMP(1,0,0,0)$ for $i =0, \ldots, n-1$. 
Thus $mpp^{(1,0,0,0)}(\sg) \geq n$. We now proceed by induction 
to prove that $B_{2n+1}(x)|_{x^n} = (2n)!!$ for all $n \geq 1$. This 
is clear for $n=1$ since $B_3(x) = 2x$.
Now suppose that $\sg = \sg_1 \ldots \sg_{2n+1} \in UD_{2n+1}$ and  
$mpp^{(1,0,0,0)}(\sg)=n$. It is then easy to see that it 
must be the case that $\sg_2 = 2n+1$. Moreover if, 
$\tau = \red[\sg_3 \ldots \sg_{2n+1}]$, then $\tau \in UD_{2n-1}$ and 
$mmp^{(1,0,0,0)}(\tau) =n-1$.  Thus since we are assuming by 
induction that $B_{2n-1}(x)|_{x^{n-1}} = (2n-2)!!$, we have 
$2n$ choices of $\sg_1$ and $(2n-2)!!$ choices for  
$\tau$.  Hence $B_{2n+1}(x)|_{x^{n}} = (2n)!!$ for $n \geq 1$.

For (3), note that if $\sg = \sg_1 \ldots \sg_{2n} \in DU_{2n}$, then 
$\sg_{2i}$ matches the pattern $MMP(1,0,0,0)$ for $i =1, \ldots, n-1$. 
Thus $mpp^{(1,0,0,0)}(\sg) \geq n-1$. 
Suppose that $mpp^{(1,0,0,0)}(\sg)=n-1$. It is then easy to see that it 
must be the case that $\sg_1 = 2n$. Moreover, if
$\tau = \sg_2 \ldots \sg_{2n}$, then $\tau \in UD_{2n-1}$ and 
$mmp^{(1,0,0,0)}(\tau) =n-1$.  Thus  we have 
$(2(n-1))!!$ choices for  
$\tau$ by part (2).  Hence $C_{2n}(x)|_{x^{n-1}} = (2(n-1))!!$.

For (4), note that if $\sg = \sg_1 \ldots \sg_{2n+1} \in DU_{2n+1}$, then 
$\sg_{2i}$ matches $MMP(1,0,0,0)$ for $i =1, \ldots, n$. 
Thus $mpp^{(1,0,0,0)}(\sg) \geq n$. 
Suppose that $mpp^{(1,0,0,0)}(\sg)=n$. It is then easy to see that it 
must be the case that $\sg_1 = 2n+1$. Moreover, if 
$\tau = \sg_2 \ldots \sg_{2n+1}$, then $\tau \in UD_{2n}$ and 
$mmp^{(1,0,0,0)}(\tau) =n$.  Thus  we have 
$(2n-1)!!$ choices for  
$\tau$ by part (1).  Hence $D_{2n+1}(x)|_{x^{n}} = (2n-1)!!$ for $n \geq 1$. 
\end{proof}

We can easily explain the coefficients of the highest power of 
$x$ in each of the polynomials $A_{2n}(x)$, $B_{2n+1}(x)$, $C_{2n}(x)$, 
and $D_{2n+1}(x)$.  That is, we have the following proposition.

\begin{proposition}\label{highest} \

\begin{itemize}
\item[\rm{(1)}] For all $n \geq 1$, the highest power of $x$ that appears in 
$A_{2n}(x)$ is $x^{2n-1}$ which appears with coefficient 
$B_{2n-1}(1)$. 

\item[\rm{(2)}]  For all $n \geq 1$, the highest power of $x$ that appears in 
$B_{2n+1}(x)$ is $x^{2n-1}$ which appears with coefficient 
$(2n)B_{2n-1}(1)$. 

\item[\rm{(3)}]  For all $n \geq 1$, the highest power of $x$ that appears in 
$C_{2n}(x)$ is $x^{2n-2}$ which appears with coefficient 
$(2n-1)A_{2n-2}(1)$. 

\item[\rm{(4)}]  For all $n \geq 1$, the highest power of $x$ that appears in 
$D_{2n+1}(x)$ is $x^{2n}$ which appears with coefficient 
$A_{2n}(1)$.
\end{itemize}
\end{proposition}
\begin{proof}
For (1), it is easy to see  that $mmp^{(1,0,0,0)}(\sg)$ is maximized for 
a $\sg = \sg_1 \ldots \sg_{2n} \in UD_{2n}$ when $\sg_{2n}=2n$. 
In such a case $mmp^{(1,0,0,0)}(\sg) = 2n-1$ and 
$\sg_1 \ldots \sg_{2n-1}$ can be any element of $UD_{2n-1}$. 

For (2), it is easy to see that $mmp^{(1,0,0,0)}(\sg)$ is maximized for 
a $\sg = \sg_1 \ldots \sg_{2n+1} \in UD_{2n+1}$ when $\sg_{2n}=2n+1$. 
In such a case $mmp^{(1,0,0,0)}(\sg) = 2n-1$. We then 
have $2n$ choices for $\sg_{2n+1}$  and 
$\red[\sg_1 \ldots \sg_{2n-1}]$ can be any element of $UD_{2n-1}$. 
Thus $B_{2n+1}(x)|_{x^{2n-1}} = (2n)B_{2n-1}(1)$.

For (3), it is easy to see  that $mmp^{(1,0,0,0)}(\sg)$ is maximized for 
a $\sg = \sg_1 \ldots \sg_{2n} \in DU_{2n}$ when $\sg_{2n-1}=2n$. 
In such a case $mmp^{(1,0,0,0)}(\sg) = 2n-2$. 
We then 
have $2n-1$ choices for $\sg_{2n}$  and 
$\red[\sg_1 \ldots \sg_{2n-2}]$ can be any element of $DU_{2n-2}$. 
Thus $C_{2n}(x)|_{x^{2n-2}} = (2n-1)C_{2n-2}(1) = (2n-1)A_{2n-2}(1)$.

For (4), it is easy to see that $mmp^{(1,0,0,0)}(\sg)$ is maximized for 
a $\sg = \sg_1 \ldots \sg_{2n+1} \in DU_{2n+1}$ when $\sg_{2n+1}=2n+1$. 
In such a case $mmp^{(1,0,0,0)}(\sg) = 2n$. Then 
$\sg_1 \ldots \sg_{2n}$ can be any element of $DU_{2n}$. 
Thus $D_{2n+1}(x) |_{x^{2n}} = C_{2n}(1) = A_{2n}(1)$.

\end{proof}

\subsection{Recursions on up-down permutations of even length}

By Theorem \ref{thm:base}, the lowest power of $x$ that appears 
with a non-zero coefficient in $A_{2n}(x)$ is $x^n$.  Next we 
consider $A_{2n}(x)|_{x^{n+k}}$ for fixed $k$. That is, 
we let 
$$A_{2n}^{=n+k} = |\{\sg \in UD_{2n}:mmp^{(1,0,0,0)}(\sg) =n +k\}|$$
for fixed $k \geq 1$. 
Our goal is to show that $A_{2n}^{=n+k} = p_k(n) (2n-1)!!$ for some 
fixed polynomial $p_k(n)$ in $n$.  That is, we shall prove 
the following theorem, where we let $$(x)\downarrow_n = 
x(x-1) \cdots (x-n+1) \mbox{ if } n \geq 1 \mbox{ and } (x)\downarrow_0 =1.$$

 \begin{theorem}\label{thm:A+k}
There is a sequence of polynomials $p_0(x),p_1(x), \ldots $ 
such that for all $k \geq 0$,  
\begin{equation*}\label{eq:A+k}
A_{2n}^{=n+k} = p_k(n)(2n-1)!! \ \mbox{for all }n \geq k+1.
\end{equation*}
Moreover for $k \geq 1$, the values $p_k(n)$ are defined by the recursion 
\begin{equation}\label{Apkrec}
p_k(n) = \frac{B_{2k+1}(1)}{(2k+1)!!}+ \sum_{j=1}^k \sum_{t=k+2}^n 
\frac{B_{2j+1}(1)2^j(t-1)\downarrow_j}{(2j+1)!}p_{k-j}(t-j-1)
\end{equation}
where $p_0(x) =1$.
\end{theorem}
\begin{proof}
We proceed by induction on $k$. For $k=0$, we know 
by Theorem \ref{thm:base} that $A_{2n}^{=n} = (2n-1)!!$ for all 
$n \geq 1$ so that we can let $p_0(x) =1$. 

Now assume that $k \geq 1$ and the theorem is true for $s < k$.
That is, assume that for $0 \leq s <k$, there is a polynomial 
$p_s(x)$ such that for $n \geq s+1$, 
$A_{2n}^{=n+s}= p_s(n)(2n-1)!!$.

It is easy to see that for $\sg = \sg_1 \ldots \sg_{2n} \in UD_{2n}$, 
$mmp^{(1,0,0,0)}(\sg) > n+k$ if $\sg_{2j}=2n$ where  $j \geq  k+2$ 
because then $\sg_2, \sg_4,\ldots, \sg_{2k+2}$ as well as 
$\sg_{2i+1}$ such that $i =0, \ldots, n-1$ will match the pattern 
$MMP(1,0,0,0)$ in $\sg$. Thus if 
$mmp^{(1,0,0,0)}(\sg) = n+k$, then 
$2n \in \{\sg_2,\sg_4, \ldots, \sg_{2k+2}\}$. 
Now suppose that $j \leq k+1$ and $\sg_{2j} =2n$. 
Then we have $\binom{2n-1}{2j-1}$ ways to choose the elements 
$\sg_1, \ldots, \sg_{2j-1}$ and we have $B_{2j-1}(1)$ ways 
to order them. Then we know that 
$\sg_i$ matches the marked mesh pattern $MMP(1,0,0,0)$ in 
$\sg$ for $i$ odd and for $i \in \{2,4, \ldots, 2j-2\}$. 
Hence, it must be the case that 
$mmp^{(1,0,0,0)}(\red[\sg_{2j+1} \ldots \sg_{2n}]) = 
n-j+k-j+1$.  Thus it follows that for $n \geq k+1$, 
\begin{equation}\label{Ankrec}
A_{2n}^{=n+k} = \sum_{j=1}^{k+1} \binom{2n-1}{2j-1} 
B_{2j-1}(1)A_{2(n-j)}^{=(n-j)+k-j+1}.
\end{equation}
Now define 
$p_k(n) = \frac{A_{2n}^{=n+k}}{(2n-1)!!}$ for $n \geq k+1$. 
Note that $A_{2k+2}^{(k+1)+k} = B_{2k+1}(1)$ since 
for a $\tau = \tau_1 \ldots \tau_{2k+2} \in UD_{2k+2}$ to 
have $mmp^{(1,0,0,0)}(\tau) =2k+1$, it must be the case 
that $\tau_{2k+2}=2k+2$ and, hence, we have 
$  B_{2k+1}(1)$ choices for $\tau_1 \ldots \tau_{2k+1}$. Hence, 
$p_k(k+1) = \frac{ B_{2k+1}(1)}{(2k+1)!!}$.

We can rewrite (\ref{Ankrec}) as 
\begin{eqnarray}\label{2Ankrec}
p_k(n) (2n-1)!! &=& (2n-1) p_k(n-1) (2n-3)!! + \nonumber \\
&& \sum_{j=2}^{k+1} \frac{\prod_{i=0}^{2j-2} (2n-1-j)}{(2j-1)!}
B_{2j-1}(1) p_{k-j+1}(n-j) 
(2n-2j-1)!!.
\end{eqnarray}
Dividing (\ref{2Ankrec}) by $(2n-1)!!$, we obtain that 
\begin{eqnarray*}\label{3Ankrec}
p_k(n) -  p_k(n-1)  &=& 
\sum_{j=2}^{k+1} \frac{B_{2j-1}(1)\prod_{s=1}^{j-1} (2n-2s)}{(2j-1)!}p_{k-j+1}(n-j) \nonumber \\
&=& \sum_{j=1}^k \frac{B_{2j+1}(1)2^j (n-1)\downarrow_j}{(2j+1)!}p_{k-j}
(n-j-1).
\end{eqnarray*}
Hence for $n \geq k+1$,
\begin{eqnarray*}  
p_k(n) -p_k(k+1) &=& \sum_{t=k+2}^n p_k(t)-p_k(t-1) \\
&=&  \sum_{t=k+2}^n 
\sum_{j=1}^k \frac{B_{2j+1}(1) 2^j (t-1)\downarrow_j}{(2j+1)!} p_{k-j}(t-j-1).
\end{eqnarray*}
It follows that for $n \geq k+1$, 
\begin{equation*}\label{4Ankrec}
p_k(n) = \frac{ B_{2k+1}(1)}{(2k+1)!!}+ \sum_{j=1}^k \sum_{t=k+2}^n  \frac{B_{2j+1}(1)2^j (t-1)\downarrow_j}{(2j+1)!}p_{k-j}(t-j-1).
\end{equation*}
This proves (\ref{Apkrec}).

Since $p_s(x)$ is a polynomial for $s <k$, it is easy to see 
that 
$$\sum_{t=k+2}^n  \frac{B_{2j+1}(1)2^j (t-1)\downarrow_j}{(2j+1)!}p_{k-j}(t-j-1)$$ is a polynomial in $n$ for $j=1, \ldots ,k$.  Thus 
$p_k(n)$ is a polynomial in $n$. 
\end{proof} 

One can use Mathematica and (\ref{Apkrec}) to compute the first 
few expressions for $p_k(n)$.  For example, we have computed 
that 
\begin{eqnarray*}
p_0(n) &=& 1, \\
p_1(n) &=& \frac{2}{3} \binom{n}{2}, \\
p_2(n) &=& \frac{n(2+7n-14n^2+5n^3)}{90}, \ \mbox{and} \\
p_3(n) &=& \frac{n(192-478n+213n^2+227n^3-198n^4+35n^5)}{5670}.
\end{eqnarray*}

\subsection{Recursions on up-down permutations of odd length}

\begin{theorem}\label{thm:B+k}
There is a sequence of polynomials $q_0(x),q_1(x), \ldots $ 
such that for all $k \geq 0$,  
\begin{equation*}\label{eq:B+k}
B_{2n+1}^{=n+k} = q_k(n)(2n)!! \ \mbox{for all }n \geq k+1.
\end{equation*}
Moreover for $k \geq 1$, the values $q_k(n)$ are defined by the recursion 
\begin{equation}\label{Bpkrec}
q_k(n) = \frac{B_{2k+1}(1)}{(2k)!!}+ \sum_{j=1}^k \sum_{t=k+2}^n 
\frac{B_{2j+1}(1)\prod_{s=0}^{j-1} (2t-1-2s)}{(2j+1)!}q_{k-j}(t-j-1)
\end{equation}
where $q_0(x) =1$.
\end{theorem}
\begin{proof}
We proceed by induction on $k$. For $k=0$, we know 
by Theorem \ref{thm:base} that $B_{2n+1}^{=n} = (2n)!!$ for all 
$n \geq 1$ so that we can let $q_0(x) =1$. 

Now assume that $k \geq 1$ and the theorem is true for $s < k$.
That is, assume that for $0 \leq s <k$, there is a polynomial 
$q_s(x)$ such that for $n \geq s+1$, 
$B_{2n+1}^{=n+s}= q_s(n)(2n)!!$.

We can argue as in Theorem \ref{thm:A+k} that  if 
$mmp^{(1,0,0,0)}(\sg) = n+k$, then 
$2n \in \{\sg_2,\sg_4, \ldots, \sg_{2k+2}\}$. 
Now suppose that $j \leq k+1$ and $\sg_{2j} =2n$. 
Then we have $\binom{2n}{2j-1}$ ways to choose the elements 
$\sg_1, \ldots, \sg_{2j-1}$ and we have $B_{2j-1}(1)$ ways 
to order them. Then we know that 
$\sg_i$ matches the marked mesh pattern $MMP(1,0,0,0)$ in 
$\sg$ for $i \in \{2,4, \ldots, 2j-2\} \cup \{1,3,\ldots,2n-1\}$. Hence, 
it must be the case that 
 $mmp^{(1,0,0,0)}(\red[\sg_{2j+1} \ldots \sg_{2n+1}]) = 
n-j+k-j+1$.  Thus it follows that for $n \geq k+2$, 
\begin{equation}\label{Bnkrec}
B_{2n+1}^{=n+k} = \sum_{j=1}^{k+1} \binom{2n}{2j-1} 
B_{2j-1}(1)B_{2(n-j)+1}^{=(n-j)+k-j+1}.
\end{equation}
Now define 
$q_k(n) = \frac{B_{2n+1}^{=n+k}}{(2n)!!}$ for $n \geq k+1$. 
Note that $B_{2k+3}^{(k+1)+k} = (2k+2)B_{2k+1}(1)$ since 
for a $\tau = \tau_1 \ldots \tau_{2k+3} \in UD_{2k+3}$ to 
have $mmp^{(1,0,0,0)}(\tau) =2k+1$, it must be the case 
that $\tau_{2k+2}=2k+3$ and, hence, we have $2k+2$ choices 
for $\tau_{2k+3}$ and 
$  B_{2k+1}(1)$ choices for $\tau_1 \ldots \tau_{2k+1}$. Thus, 
$q_k(k+1) = \frac{ (2k+2)B_{2k+1}(1)}{(2k+2)!!} = \frac{B_{2k+1}(1)}{(2k)!!}$.

We can rewrite (\ref{Bnkrec}) as 
\begin{eqnarray}\label{2Bnkrec}
q_k(n) (2n)!! &=& (2n) q_k(n-1) (2n-2)!! + \nonumber \\
&&\sum_{j=2}^{k+1} \frac{\prod_{i=0}^{2j-2} (2n-j)}{(2j-1)!}
B_{2j-1}(1) q_{k-j+1}(n-j) 
(2n-2j)!!.
\end{eqnarray}
Dividing (\ref{2Bnkrec}) by $(2n)!!$, we obtain that 
\begin{equation}\label{3Bnkrec}
q_k(n) -  q_k(n-1)  = 
\sum_{j=2}^{k+1} \frac{B_{2j-1}(1)\prod_{s=1}^{j-1} (2n-2s-1)}{(2j-1)!}q_{k-j+1}(n-j).
\end{equation}
Hence for $n \geq k+2$,
\begin{eqnarray*}  
q_k(n) -q_k(k+1) &=& \sum_{t=k+2}^n q_k(t)-q_k(t-1) \\
&=&  \sum_{t=k+2}^n 
\sum_{j=1}^k \frac{B_{2j+1}(1) 2^j \prod_{s=1}^{j-1} (2n-2s-1)}{(2j+1)!} q_{k-j}(t-j-1).
\end{eqnarray*}
It follows that for $n \geq k+1$, 
\begin{equation}\label{4Bnkrec}
q_k(n) = \frac{B_{2k+1}(1)}{(2k)!!}+ \sum_{j=1}^k \sum_{t=k+2}^n  \frac{B_{2j+1}(1)2^j \prod_{s=1}^{j-1} (2n-2s-1)}{(2j+1)!}q_{k-j}(t-j-1).
\end{equation}
This proves (\ref{Bpkrec}).

Since $q_s(x)$ is a polynomial for $s <k$, it is easy to see 
that \\
$\sum_{t=k+2}^n  \frac{B_{2j+1}(1)\prod_{s=1}^{j-1} (2n-2s-1)}{(2j+1)!}q_{k-j}(t-j-1)$ is a polynomial in $n$ for $j=1, \ldots ,k$.  Thus 
$q_k(n)$ is a polynomial in $n$. 
\end{proof} 

One can use Mathematica and (\ref{Bpkrec}) to compute the first 
few examples of $q_k(n)$.  For example, we have computed 
that 
\begin{eqnarray*}
q_0(n) &=& 1, \\
q_1(n) &=& \frac{n^2-1}{3}, \\
q_2(n) &=& \frac{(n-2)(n-1)(5n^2+n-3)}{90}, \ \mbox{and} \\
q_3(n) &=& \frac{35n^6-84n^5-193n^4+345n^3+140n^2-81n+198}{5670}.
\end{eqnarray*}

\subsection{Recursions on down-up permutations}

Similar results hold for down-up permutations. 

\begin{theorem}\label{thm:CD+k}
There are sequences of polynomials $r_0(x),r_1(x), \ldots $ and 
$s_0(x),s_1(x), \ldots $
such that for all $k \geq 0$,  
\begin{equation}\label{eq:C+k}
C_{2n}^{=n-1+k} = r_k(n)(2n-2)!! \ \mbox{for all }n \geq k+1.
\end{equation}
and 
\begin{equation}\label{eq:D+k}
D_{2n+1}^{=n-1+k} = s_k(n)(2n-1)!! \ \mbox{for all }n \geq k+1.
\end{equation}
\end{theorem}
\begin{proof}
By Theorem \ref{thm:base},  
$C_{2n}^{= n-1} =(2n-2)!!$ and $D_{2n+1}^{= n} =(2n-1)!!$ for all $n \geq 1$. 
Thus we can let $r_0(x) = s_0(x) =1$.

For a permutation $\sg = \sg_1 \ldots \sg_{2n} \in DU_{2n}$ to have 
$mmp^{(1,0,0,0)}(\sg) = n-1+k$, it must be the case that 
$2n \in \{\sg_1, \sg_3, \ldots, \sg_{2k+1}\}$.  If $\sg_{2j+1}=2n$ 
where $j \in \{0,1, \ldots, k\}$, then there are $\binom{2n-1}{2j}$ ways to pick 
the elements of $\sg_1 \ldots \sg_{2j}$ and $C_{2j}(1)$ ways to order 
them.  Then $\red[\sg_{2j+2} \ldots \sg_{2n}] \in UD_{2(n-j-1)+1}$ and 
must have $n-1+k-(2j)$ matches of $MMP(1,0,0,0)$. Thus we have 
$B_{2(n-j-1)+1}^{=(n-j-1)+k-j}$ ways to order $\sg_{2j+2} \ldots \sg_{2n}$. 
It follows that for $n \geq k+1$, 
\begin{equation}\label{Cnkrec1}
C_{2n}^{=n-1+k} = \sum_{j=0}^k \binom{2n-1}{2j} C_{2j}(1) B_{2(n-j-1)+1}^{
= n-j-1 +k-j}.
\end{equation}
But $C_{2j}(1) = A_{2j}(1)$ and $B_{2(n-j-1)+1}^{=n-j-1 +k-j} = 
(2(n-j-1))!! q_{k-j}(n-j-1)$. Thus for $n \geq k+1$, 
\begin{eqnarray*}\label{Cnkrec2}
C_{2n}^{=n-1+k} &=& \sum_{j=0}^k \binom{2n-1}{2j}A_{2j}(1) 
(2(n-j-1))!! q_{k-j}(n-j-1) \nonumber \\
&=& (2n-2)!! \sum_{j=0}^k \frac{A_{2j}(1) \prod_{s=1}^j (2n+1-2s)}{(2j)!}  
(2(n-j-1))!! q_{k-j}(n-j-1).
\end{eqnarray*}
Thus $C_{2n}^{=n-1+k} = (2n-2)!!r_k(n)$ where 
\begin{equation}\label{rkdef}
r_k(n) = \sum_{j=0}^k \frac{A_{2j}(1) \prod_{s=1}^j (2n+1-2s)}{(2j)!} q_{k-j}(n-j-1).
\end{equation}

A similar argument will show that for $n \geq k+1$,
\begin{equation*}\label{Dnkrec1}
D_{2n+1}^{= n+k} = \sum_{j=0}^k \binom{2n}{2j}C_{2j}(1) A_{2(n-j)}^{=n-j+k-j}.
\end{equation*}
Since $A_{2(n-j)}^{= n-j+k-j} = (2(n-j)-1)!!p_{k-j}(n-j)$, we obtain 
that 
\begin{eqnarray*}\label{Dnkrec2}
D_{2n+1}^{= n+k} &=& \sum_{j=0}^k \binom{2n}{2j} A_{2j}(1) 
(2(n-j)-1)!! p_{k-j}(n-j) \nonumber \\
&=& (2n-1)!! \sum_{j=0}^k \frac{A_{2j}(1) \prod_{s=1}^j (2n+2-2s)}{(2j)!} 
p_{k-j}(n-j).
\end{eqnarray*}
Thus $D_{2n+1}^{=n+k} = (2n-2)!!s_k(n)$ where 
\begin{equation}\label{skdef}
s_k(n) = \sum_{j=0}^k \frac{A_{2j}(1) \prod_{s=1}^j (2n+2-2s)}{(2j)!} 
p_{k-j}(n-j).
\end{equation}
\end{proof}

One can use (\ref{rkdef}) and (\ref{skdef}) to compute $r_k(n)$ and $s_k(n)$ 
for the first few values of $k$.  For example, we have 
that 
\begin{eqnarray*}
r_0(n) &=& 1, \\
r_1(n) &=& \frac{2n^2+2n-3}{6}, \\
r_2(n) &=& \frac{20n^4+24n^3 -128n^2-12n+45}{360}, \ \mbox{and} \\
r_3(n) &=& \frac{280n^6+168n^5-4820n^4 +3168n^3+8734n^2-6702n+2835}{45360}.
\end{eqnarray*}
Similarly, we have 
\begin{eqnarray*}
s_0(n) &=& 1, \\
s_1(n) &=& \frac{n(n+2)}{3}, \\
s_2(n) &=& \frac{n(5n^3 +16n^2-68n+47)}{90}, \ \mbox{and} \\
s_3(n) &=& \frac{n(35n^5+126n^4 -340n^3-417n^2+656n-60)}{5760}.
\end{eqnarray*}

\section{Conclusion}

In this paper, we have shown that one can find 
the generating functions for the distribution of 
the quadrant 
marked mesh patterns $MMP(1,0,0,0)$,  $MMP(0,1,0,0)$, $MMP(0,0,1,0)$, 
and $MMP(0,0,0,1)$ in both up-down and down-up permutations by 
proving simple recursions based on the position of the largest 
element in a permutation.  As noted in Subsection \ref{sub2-5}, 
these simple type of recursions no longer hold for the distribution 
of the quadrant 
marked mesh patterns $MMP(k,0,0,0)$,  $MMP(0,k,0,0)$, $MMP(0,0,k,0)$, 
and $MMP(0,0,0,k)$ in both up-down and down-up permutations when 
$k \geq 2$.  However, our techniques can be used to study 
the distribution of other quadrant marked mesh patterns in 
up-down and down-up permutations. For example, in \cite{kitrem2}, 
we have proved similar recursions based on the position of the smallest 
element in a  permutation to 
study the distribution of the quadrant marked mesh patterns 
$MMP(1,0,\emptyset,0)$,  $MMP(0,1,0,\emptyset)$, $MMP(\emptyset,0,1,0)$, 
and $MMP(0,\emptyset,0,1)$ in both up-down and down-up permutations. 
In this case, the recursions are a bit more subtle and the 
corresponding generating functions  are not always as simple  
as the results of this paper. 
For example, if 
we let 
\begin{eqnarray*}
A^{(1,0,\emptyset,0)}(x,t) &=& 
1+ \sum_{n \geq 1} \frac{t^{2n}}{(2n)!}\sum_{\sg \in UD_{2n}} 
x^{\mmp^{(1,0,\emptyset,0)}(\sg)} \ \mbox{and} \\
B^{(1,0,\emptyset,0)}(x,t) &=& 
\sum_{n \geq 0} \frac{t^{2n+1}}{(2n+1)!}\sum_{\sg \in UD_{2n+1}} 
x^{\mmp^{(1,0,\emptyset,0)}(\sg)},
\end{eqnarray*}
then we can show that 
\begin{eqnarray*}
A^{(1,0,\emptyset,0)}(t,x) &=& (\sec(t))^x, \\
B^{(1,0,\emptyset,0)}(t,x) &=& 
\frac{\sin(t)\cos(t)(1-x+\sec(t))}{x+(1-x)\cos(t)} \times \\
&& \ \ \ \left( 
(1-x)\ {}_2F_1 \left(\frac{1}{2},\frac{1+x}{2};\frac{3}{2};\sin\left(t^2\right)\right) + 
x\  {}_2F_1 \left(\frac{1}{2},\frac{2+x}{2};\frac{3}{2};\sin\left(t^2\right)\right)\right) 
\end{eqnarray*}
where ${}_2F_1(a,b;c;z)= \sum_{n=0}^\infty \frac{(a)_n(b)_n}{(c)_n} 
\frac{z^n}{n!}$ and 
$(x)_n = x(x-1) \cdots (x-n+1)$ if $n \geq 1$ and $(x)_0 =1$.

There are several directions for further research that are 
suggested by the results of this paper. First, one 
can study the distribution in up-down and down-up permutations 
of other quadrant marked meshed patterns $MMP(a,b,c,d)$ in 
the case where $a,b,c,d \in \{\emptyset,1\}$. More generally,  
one can study the distribution of 
quadrant marked mesh patterns on other classes of pattern-restricted permutations such as 2-{\em stack-sortable permutations} or {\em vexillary permutations}
(see \cite{kit} for definitions of these) and many other permutation classes having nice properties. Finally, we conjecture  
that the polynomials $A_{2n}(x)$, $B_{2n+1}(x)$, $C_{2n}(x)$, and 
$D_{2n+1}(x)$ are unimodal for all $n \geq 1$. This is certainly true 
for small values of $n$.

\end{document}